\documentclass{amsart} 
\usepackage{amssymb}
\usepackage{amsthm}
\usepackage[curve,matrix,arrow]{xy}
\theoremstyle{plain}\newtheorem{Theorem}{Theorem}[section]
\theoremstyle{plain}
\theoremstyle{plain}
\theoremstyle{plain}\newtheorem{Lemma}[Theorem]{Lemma}
\theoremstyle{plain}\newtheorem{Proposition}[Theorem]{Proposition}
\theoremstyle{definition}
\theoremstyle{definition}
\theoremstyle{definition}
\theoremstyle{definition}\newtheorem{Remark}[Theorem]{Remark}
\theoremstyle{definition}
\theoremstyle{plain}\newtheorem{Equation}[Theorem]{}

\def\CB{{\mathcal{B}}}      
\def\CC{{\mathcal{C}}}    
\def\CD{{\mathcal{D}}}    
    
\def\CF{{\mathcal{F}}}    
\def\CG{{\mathcal{G}}}

\def\CO{{\mathcal{O}}}

\def\bS{{\mathbb{S}}}
\def\bT{{\mathbb{T}}}

           \def\tenk{\otimes_k}     
             \def\ten{\otimes} \def\tenR{\otimes_R}
\def\chr{\mathrm{char}}

\def\dim{\mathrm{dim}}           
\def\End{\mathrm{End}}           
\def\Endbar{\underline{\mathrm{End}}}
\def\Ext{\mathrm{Ext}}           

\def\hatExt{\widehat{\mathrm{Ext}}} 
\def\barExt{\overline{\mathrm{Ext}}}
\def\HH{H\!H}
\def\hatHH{\widehat{H\!H}} \def\barHH{\overline{H\!H}}
\def\Hom{\mathrm{Hom}}           
\def\Hombar{\underline{\mathrm{Hom}}}
           
\def\Id{\mathrm{Id}}             \def\tenA{\otimes_A}
             \def\tenB{\otimes_B}
           \def\tenK{\otimes_K}

           \def\tenK{\otimes_K}
\def\Mod{\mathrm{Mod}}           \def\tenO{\otimes_{\mathcal{O}}}
\def\mod{\mathrm{mod}}      \def\modbar{\underline{\mathrm{mod}}}
\def\op{\mathrm{op}}
\def\perf{\mathrm{perf}}    \def\perfbar{\underline{\mathrm{perf}}}
\def\pr{\mathrm{pr}}

\def\rk{\mathrm{rk}}

\def\RHom{\mathrm{RHom}}

\def\tr{\mathrm{tr}}             
 \def\trace{\mathrm{trace}}

\title[Tate duality and transfer over discrete valuation rings]{Tate duality  and
transfer for symmetric algebras
over complete discrete valuation rings} 
\author{Markus Linckelmann} 
\date{}

\begin{document}

\begin{abstract}
We show that dualising transfer maps in Hochschild cohomology
of symmetric algebras over complete discrete valuations rings commutes 
with Tate duality. 
This is analogous to a similar result for Tate cohomology of symmetric algebras
over fields. We interpret both results in the broader context of Calabi-Yau
triangulated categories.
\end{abstract}

\maketitle

\section{Introduction}

An algebra $A$ over a commutative ring $R$ is called {\em symmetric} if it is
finitely generated projective as  an $R$-module and if $A\cong A^\vee$ as
$A$-$A$-bimodules, where $A^\vee=\Hom_R(A,R)$. In that case, the image
$s$ in $A^\vee$  of $1_A$ under such a bimodule isomorphism is called a 
{\em symmetrising form for $A$}. The form $s$ depends on the choice of
the isomorphism $A\cong A^\vee$, and is unqiue up to multiplication by
an invertible element in $Z(A)$. There may not be a canonical choice
for $s$.  If $G$ is a finite group, then $RG$ is symmetric and - keeping track of
the image of $G$ in $RG$ - does have a canonical
symmetrising form, namely the map $s$ sending $1_G$ to $1_R$ and all nontrivial
group elements to $0$.

\medskip
For a symmetric algebra $A$ over a field, Tate duality
is a duality between the Tate-Ext spaces
$\hatExt_A^{n-1}(U,V)$ and $\hatExt_A^{-n}(V,U)$, for any integer $n$ and any
two finite-dimensional $A$-modules $U$, $V$. 
In particular, this yields a duality 
between  Tate-Hochschild cohomology $\hatHH^{n-1}(A)$ and $\hatHH^{-n}(A)$ 
for all integers $n$.
It is shown  in \cite{LiHHtt} that in that case Tate duality commutes with the 
transfer maps introduced in \cite{Lintransfer}, extending a well-known
compatibility of Tate duality with restriction and transfer in finite group
cohomology. 

\medskip
 If $A$ is instead a symmetric  algebra over a complete discrete valuation ring 
 $\CO$ with a separable coefficient extension $K\tenO A$ to the field of fractions
 $K$ of $\CO$, and if $U$, $V$ are $\CO$-free
finitely generated $A$-modules, then Tate duality takes a  different form:
there is a non-degenerate bilinear form 
$$\langle - , - \rangle_A : \hatExt_A^{n}(U,V)\times \hatExt_A^{-n}(V,U)
\to K/\CO$$ 
for any integer $n$, which is described explicitly in \cite{EGKL}
and briefly reviewed in \ref{Tatedegree0-3} below.
The purpose of this paper is to show that  this duality
 commutes with the transfer maps from \cite{Lintransfer}. 
The proof is quite different from that in \cite{LiHHtt} due to the different
description of Tate duality, as given in \cite{EGKL}, extending the description in 
Th\'evenaz 
\cite[\S 33]{Thev} for finite group algebras. Both this duality as well as
the transfer maps depend on the choices of symmetrising forms. By omitting
choices of symmetrising forms from the statements below we implicitly assert
that these statements hold regardless of these choices.

\medskip
If $A$, $B$ are two symmetric $\CO$-algebras 
and $M$ is an $A$-$B$-bimodule 
which is finitely generated as a left $A$-module and as a right 
$B$-module, then, for any two finitely generated $\CO$-free
$A$-modules $U$, $V$, there is a transfer map 
$$\tr_{M}(U,V) : \hatExt_B^n(M^\vee\tenA U,M^\vee\tenA V)\to
\hatExt_A^n(U,V),$$
which we will  review in Section \ref{transferSection}. 
In degree zero, this is the trace map defined in \cite[(57)]{Morita65},  \cite{Broue1} or 
\cite[Definition 6.6]{Broue2}. For transfer induced by biadjoint functors between
more general categories see \cite{Chouinard}, and
for the graded version needed in this paper see \cite[\S 4, \S 7]{Ligrblock} or also 
\cite[\S 5]{LiHHtt}, for instance. For simplicity we will write $\tr_M$ instead
of $\tr_M(U,V)$ whenever $U$, $V$ are clear from the context.

\begin{Theorem} \label{Thm1}
Let $A$, $B$ be symmetric algebras over a complete discrete valuation ring
$\CO$ with a field of fractions $K$ of characteristic zero.
 Suppose  that $K\tenO A$, $K\tenO B$ are semisimple. 
 Let $M$ be an $A$-$B$-bimodule which is finitely generated projective as a left
 $A$-module and as a right $B$-module. Let $n$ be an integer, and let $U$, $V$ 
 be finitely generated
 $\CO$-free $A$-modules. For any $\alpha\in \hatExt^n_A(U,V)$ and
 $\beta\in \hatExt_B^{-n}(M^\vee\tenA U, M^\vee\tenA V)$ we have
 $$\langle \alpha, \tr_M(\beta)\rangle_A =
  \langle \Id_{M^\vee} \ten \alpha , \beta\rangle_B,$$
  $$\langle \tr_M(\beta), \alpha\rangle_A = 
  \langle \beta, \Id_{M^\vee}\ten \alpha\rangle_B$$
  \end{Theorem} 
  
 One consequence of  Theorem \ref{Thm1}  is that Tate 
duality determines   the transfer  maps $\tr_M(U,V)$.  
This comment applies also to the transfer maps in 
 \cite[Theorem 1.2]{LiHHtt} and Tate duality for symmetric $k$-algebras.
 This has an interpretation in the  context of  Calabi-Yau triangulated 
 categories, which we will describe in  Section \ref{CY-section}.

\medskip
Tate-Ext applied to $A$ as a module over the symmetric algebra
$A^e=$ $A\tenO A^\op$ yields Tate-Hochschild cohomology $\hatHH^*(A)$ (see
\ref{hatHHDef} below). Tate duality applied to this situation yields in turn
a non-degenerate bilinear form
$$\langle -, -\rangle_{A^e} : \hatHH^n(A) \times \hatHH^{-n}(A) \to K/\CO;$$
see for instance  \cite[Remark 1.5]{EGKL}.   By \cite[Definition 2.9]{Lintransfer} or 
 the more general construction principle \cite[\S 4, \S 7]{Ligrblock} specialised
 to stable categories of bimodules, the $A$-$B$-bimodule $M$ as above 
 induces a transfer map 
 $$\tr_M : \hatHH^*(B)\to \hatHH^*(A), $$ 
 that we will review in Section \ref{transferSection}.
 The dual $M^\vee$ with respect to the  base ring is a $B$-$A$-bimodule which is
 finitely generated projective as a left $B$-module and as a right $A$-module,
 hence induces a transfer map 
$\tr_{M^\vee} : \hatHH^*(A)\to \hatHH^*(B)$.
There is some abuse of notation:
the transfer map $\tr_M$ in Tate-Hochschild cohomology is not quite a special 
case of the transfer maps $\tr_M(U,V)$ defined previously; their precise  
relationship is described in Remark \ref{notationRemark}. The compatibility
between transfer and Tate duality for Tate-Hochschild cohomology takes the
following form.

\begin{Theorem}  \label{Thm2}
Let $A$, $B$ be symmetric algebras over a complete discrete valuation ring
$\CO$ with a field of fractions $K$ of characteristic zero.
 Suppose  that $K\tenO A$, $K\tenO B$ are semisimple. 
 Let $M$ be an $A$-$B$-bimodule which is finitely generated projective as a left
 $A$-module and as a right $B$-module. Let $n$ be an integer.
 For $\zeta\in \hatHH^n(A)$ and $\tau\in \hatHH^{-n}(B)$ we have
 $$\langle \zeta, \tr_M(\tau)\rangle_{A^e}  = 
 \langle \tr_{M^\vee}(\zeta), \tau \rangle_{B^e},$$
 $$\langle \tr_M(\tau), \zeta \rangle_{A^e} = 
 \langle \tau, \tr_{M^\vee}(\zeta) \rangle_{B^e}.$$
 \end{Theorem}

\begin{Remark}
In view of the interpretation of Theorem \ref{Thm1} in terms of
Calabi-Yau triangulated categories (which we will describe in Section 
\ref{CY-section}), it is worth noting that
 there are finite-dimensional algebras which are selfinjective,  not 
 necessarily symmetric, but
 whose stable category is Calabi-Yau of nonnegative dimension. See for
 instance  \cite{ErdSkow}, \cite{Dugas12}, \cite{Ivanov13}, \cite{IvaVol14}, 
 and the references therein. 
 We further draw attention to the appendix in \cite{Bocklandt} by M. Van den Bergh 
 regarding  signs in Serre duality. We largely ignore sign issues in \S 4 (notably in  
 \ref{trM-graded-notation})  because this will not
 needed for the results of this paper, but would be needed for an in-depth 
 interpretation of these results in terms of Calabi-Yau duality. 
 \end{Remark}  

\begin{Remark} \label{motivationRemark}
The main motivation for developing this material is to extend results on finite
group cohomology to symmetric algebras, in order  to provide techniques to
calculate cohomological invariants, such as the Castelnuovo-Mumford regularity, 
that might distinguish classes of symmetric algebras from being block algebras
of finite groups (see \cite{KLcm} for calculations in this context).  
One such distinguishing feature (and necessary tool for
calculations of the regularity)  is  the existence of
a local cohomology spectral sequence for Hochschild cohomology, analogous
to Greenlees' local cohomology spectral sequence in \cite{Greenlees}. 
Benson's approach to this spectral sequence in \cite{BenNY}
makes use of the compatibility of restriction and transfer in group cohomology
with respect to Tate duality over a field. We expect that this compatibility
at the level of Hochschild cohomology, both in \cite{LiHHtt} for algebras over
fields, and the present paper for algebras over complete discrete valuation rings,
 will be one of the technical ingredients towards this programme.
\end{Remark}

\begin{Remark}
Unlike in the Tate duality for symmetric algebras over fields, there is no degree shift 
in the Tate duality for symmetric algebras over a complete discrete valuation
ring $\CO$ with field of fractions $K$. This is due to the fact that we have
replaced duality with respect to the base ring $\CO$ by duality with respect to
the injective syzygy $K/\CO$ of $\CO$, noting that 
we have a short exact sequence $0 \to \CO \to K\to K/\CO\to 0$ in which
both $K$ and $K/\CO$ are injective $\CO$-modules, with $K/\CO$ in degree 
$1$ of this injective resolution of $\CO$.  For Tate cohomology over
more general rings, see Buchweitz \cite{BuchHan}. Further extensions of Tate
cohomology can be found, for instance, in \cite{Farrell}, \cite{Goichot}.
\end{Remark}

\section{Preliminaries} \label{prelim}

We will use without further reference well-known basic material on stable module
categories of symmetric algebras; see e. g. \cite[\S 2.13]{LiBookI}. We briefly
review the main properties of shift functors on stable module categories for
symmetric algebra, 
mainly to adopt some notational abuse for simplicity of exposition later on.

Let $R$ be a commutative Noetherian ring with unit element.
Let $A$ be an $R$-algebra.  An $A$-module $U$ is called {\em relatively
$R$-projective}  if the canonical surjection of $A$-modules
$A\tenR U \to U$ sending $a\ten u$ to $au$ is split, and $U$ is called
{\em relatively $R$-injective}
if the canonical injection of $A$-modules $U\to \Hom_R(A,U)$, $u\mapsto 
(a\mapsto au)$ is split (where $u\in U$ and $a\in A$). See \cite[Section 2.6]{LiBookI}
for details. 

Assume now that $A$ is symmetric.
Then the two classes of relatively $R$-projective and relatively $R$-injective
modules coincide (cf. \cite[Theorem 2.15.1]{LiBookI}).
 We denote by $\modbar(A)$ the relatively $R$-stable
category of finitely generated $A$-modules. The objects of
$\modbar(A)$ are the finitely generated $A$-modules, and morphisms
in $\modbar(A)$ are classes of $A$-homomorphisms $\Hombar_A(U,V)=$
$\Hom_A(U,V)/\Hom^\pr_A(U,V)$, where $U$, $V$ are finitely generated 
$A$-modules and $\Hom^\pr_A(U,V)$ is the $R$-module of $A$-homomorphisms
from $U$ to $V$ which factor through a finitely generated relatively
$R$-projective $A$-module. 
Composition in $\modbar(A)$ is induced by that in the category
of finitely generated $A$-modules $\mod(A)$. Since $R$ is Noetherian, the
category $\mod(A)$ is a full abelian subcategory of the category $\Mod(A)$ of
all $A$-modules. The category $\modbar(A)$ is no longer abelian but
triangulated (cf. \cite[Section A.3]{LiBookII}),with shift functor $\Sigma_A$ 
which sends a finitely generated $A$-module
$U$ to the cokernel of $U\to I$ for some relatively $R$-injective envelope
$I$ of $U$. This assignment is unique up to unique isomorphism in $\modbar(A)$,
hence does indeed induce a functor on $\modbar(A)$ 
(cf. \cite[Theorem 2.14.4]{LiBookI}), still denoted $\Sigma_A$. 
Moreover, since $A$ is symmetric,  the functor $\Sigma_A$
 is an equivalence on $\modbar(A)$ through which $\modbar(A)$
becomes indeed  a triangulated category (cf. \cite[Theorem A.3.2]{LiBookII}) .  
An inverse, denoted $\Sigma_A^{-1}$, of $\Sigma_A$ is induced by the assignement 
sending a finitely generated $A$-module $U$ to the kernel of a relatively
$R$-projective cover $P\to U$ of $U$.  For any non-negative integer $n$ we define
$\Sigma^n_A(U)$ as the $n$-th cokernel of a relatively $R$-injective resolution
of $U$, and $\Sigma_A^{-n}(U)$ as the $n$-th kernel of a relatively $R$-projective
resolution of $U$. As functors on $\modbar(A)$ we have canonical isomorphisms
$\Sigma^n_A \cong (\Sigma_A)^n$ and $\Sigma_A^{-n}\cong$ $(\Sigma_A^{-1})^n$.
We adopt the  convention that $\Sigma_A^0$ is the identity functor on $\modbar(A)$.
For any two integers $n$, $m$, we have canonical identifications of functors
$\Sigma_A^n\circ \Sigma_A^m \cong \Sigma_A^{n+m}$ on $\modbar(A)$. 
For any integer $n$, we have 
\begin{Equation} \label{TateExtDef}
$$\hatExt^n_A(U,V)=\Hombar_A(U,\Sigma_A^n(V)). $$
\end{Equation}
For $n>0$ we have $\hatExt^n_A(U,V)=$ $\Ext^n_A(U,V)$. 
If $U$ is finitely generated projective as an $R$-module, then a relatively $R$-injective
envelope (resp. resolution) can be constructed by taking the $R$-dual of
a relatively $R$-projective cover (resp. resolution) of the $R$-dual of $U$.
Note that if $U$ is finitely generated projective over $R$, then a relatively $R$-projective
resolution of $U$ is in fact a projective resolution of $U$. Thus if $U$, $V$ are
two $A$-modules which are finitely generated projective as $R$-modules, then
$\Hom_A^\pr(U,V)$ consists of the space of $A$-homomorphisms from $U$ to $V$
which factor through a finitely generated projective $A$-module. 
In general, $U$ need
not have an injective resolution of finitely generated injective $A$-modules.

\medskip
We will need the fact that the functor $\Sigma_A$ on $\modbar(A)$ lifts to
an exact functor (albeit not an equivalence in general) on $\mod(A)$ as
follows: set $A^e=$ $A\tenR A^\op$, and consider $A$ as an $A^e$-module.
Then $\Sigma_{A^e}(A)$ is an $A^e$-module which is finitely generated
projective as a left and right $A$-module, and the functor
$\Sigma_{A^e}(A)\tenA -$ on $\mod(A)$ is exact and induces the equivalence
$\Sigma_A$ on $\modbar(A)$.  As an $A$-$A$-bimodule, $\Sigma^{-1}_{A^e}$ is
the kernel of the multiplication map $A\tenR A \to A$, $a\ten b\mapsto ab$, because
this is a projective cover (not necessary minimal) of $A$. The $R$-dual of the
multiplication map together with the symmetry of $A$ yields a relatively
$R$-injective envelope $A\to A\tenR A$.
Thus choosing a symmetrising form of $A$ uniquely determines in $\modbar(A^e)$ 
 an  $A^e$-isomorphism
$$\Sigma_{A^e}(A) \cong (\Sigma_{A^e}^{-1}(A))^\vee$$
and hence more generally, $A^e$-isomorphisms
$$\Sigma_{A^e}^n(A) \cong (\Sigma_{A^e}^{-n}(A))^\vee$$
for all integers $n$.  The Tate analogue  $\hatHH^*(A)$ of the Hochschild 
cohomology $\HH^*(A)$ of $A$  is 
\begin{Equation} \label{hatHHDef}
$$\hatHH^n(A) = \hatExt_{A^e}^n(A, A) = \Hombar_{A^e}(A, \Sigma^n(A)).$$
\end{Equation}
As before, for  $n>0$ we have $\hatHH^n(A)=$ $\HH^n(A)$. 

\medskip
Let $A$, $B$, $C$ be symmetric $R$-algebras. Then the $R$-algebras
$A^\op$, $A\tenR B$, and $A\tenR B^\op$ are symmetric. 
An $A$-$B$-bimodule, or equivalently, an $A\tenR B^\op$-module,  is called 
{\em perfect} if it is finitely generated
projective as a left $A$-module and as a right $B$-module.
We denote by $\perf(A,B)$ the category of perfect $A$-$B$-bimodules.
Note that all modules in $\perf(A,B)$ are
finitely generated projective as $R$-modules. The category $\perf(A,B)$ is
a full $R$-linear subcategory of $\mod(A\tenR B^\op)$ which is closed under
taking direct summands. 
We denote by $\perfbar(A, B)$ the
image of $\perf(A, B)$ in $\modbar(A\tenR B^\op)$; this is a
thick subcategory of the triangulated category $\modbar(A\tenR B^\op)$.
If $M$ is a perfect $A$-$B$-bimodule and $N$ a perfect $B$-$C$-bimodule,
then $M\tenB N$ is a perfect $A$-$C$-bimodule. In particular, the exact
functors $\Sigma^n_{A^e}(A)\tenA -$ and $-\tenB \Sigma^n_{B^e}(B)$ on
$\mod(A\tenR B^\op)$ restrict to exact functors on $\perf(A, B)$, and they
induce functors on $\modbar(A\tenR B^\op)$ which are both canonically
isomorphic to the functor $\Sigma^n_{A\tenR B^\op}$ on $\modbar(A\tenR B^\op)$.
Note that $\perf(A,A)$ is closed under the tensor product over $A$, and hence
$\perfbar(A,A)$ is a tensor triangulated category, with tensor product $-\tenA -$.

If the algebra under consideration is clear from the context, we will simply write
$\Sigma$ for the shift functor on the stable module category, and sometimes
use the same letter $\Sigma$ for some exact lift to the category of finitely
generated modules. This is to keep notation under control, but requires some
care when it comes to establishing that all constructions are well-defined.

\section{Adjunction  for symmetric algebras}

We briefly review without proofs some formalities on bimodules over symmetric 
algebras; broader expositions can be found in many sources such as \cite{Broue1}, 
\cite{Broue2},
\cite[\S 6 Appendix]{Lintransfer},  \cite[\S 3]{LiHHtt}, \cite[\S 2.12]{LiBookI}.
Let $R$ be a commutative Noetherian ring (with unit element), and let $A$, $B$
be symmetric $R$-algebras, with symmetrising forms $s$ and $t$, respectively.
The functor $\Hom_R(-,R)$ is contravariant, and 
for $U$ and $R$-module, we write $U^\vee=\Hom_A(U,R)$.
Let $M$ be a perfect $A$-$B$-bimodule. Since $A$, $B$ are symmetric, the $R$-dual
$M^\vee$ is a perfect $B$-$A$-bimodule. We have a $B$-$A$-bimodule isomorphism
\begin{Equation} \label{MAMvee}
$$\Hom_A(M,A)\cong M^\vee$$ 
\end{Equation}
sending $\alpha\in$ $\Hom_A(M,A)$ to $s\circ \alpha$,
and we have a $B$-$A$-bimodule isomorphism
\begin{Equation}
$$\Hom_{B^\op}(M,B)\cong M^\vee$$ 
\end{Equation}
sending $\beta\in $ $\Hom_{B^\op}(M,B)$ to $t\circ\beta$.
For any $A$-module $U$ we have natural isomorphisms
\begin{Equation} \label{RAdualten} 
$$M^\vee\tenA U \cong \Hom_A(M,A)\tenA U \cong \Hom_A(M,U)$$
\end{Equation} 
where the first map is induced by the isomorphism from equation \ref{MAMvee}
and the  second map sends $\lambda\ten u $ to the map $m\mapsto \lambda(m)u$,
for $u\in U$, $m\in M$, $\lambda \in \Hom_A(M,A)$. Using that $M$ is finitely
generated projective as an $A$-module one sees that this is indeed an isomorphism.
Combining this with the tensor-Hom adjunction shows that
the functors $M\tenB-$ and $M^\vee\tenA-$
between the categories $\mod(A)$ and $\mod(B)$ of finitely generated modules
over $A$ and $B$, respectively, are left and  right adjoint to each other.
More precisely, the choices of symmetrising forms $s$, $t$ determine adjunction 
isomorphisms as follows.  For $U$ a finitely generated $A$-module and $V$ a 
finitely generated $B$-module, we have a natural isomorphism
\begin{Equation} \label{MMveeadj}
$$\Hom_A(M\tenB V,U) \cong \Hom_B(V,M^\vee\tenA U)$$
\end{Equation}
sending $\lambda_{\gamma,u}$ to the map $v\mapsto$ $s\circ\gamma_v\ten u$,
where $\gamma\in$ $\Hom_A(M,A)$, $u\in$ $U$, where
$\lambda_{\gamma,u}\in$ $\Hom_A(M\tenB V,U)$ is defined by
$\lambda_{\gamma,u}(m\ten v)=$ $\gamma(m\ten v)u$, and where 
$\gamma_v\in$ $\Hom_A(M,A)$ is defined by
$\gamma_v(m)=$ $\gamma(m\ten v)$, for all $m\in$ $M$, $v\in$ $V$.
The unit and counit of this adjunction are represented by
bimodule homomorphisms 
\begin{Equation} \label{MMveeadjunit}
$$\epsilon_M : B \longrightarrow M^\vee\tenA M\ ,\ \ 
1_B \mapsto \sum_{i\in I}\ (s\circ \alpha_i) \ten m_i\ ,$$
$$\eta_M : M\tenB M^\vee\longrightarrow A\ ,\ \ m\ten (s\circ\alpha)\mapsto
\alpha(m)\ ,$$
\end{Equation}
where $I$ is a finite indexing set,
$\alpha_i\in$ $\Hom_A(M,A)$ and $m_i\in$ $M$ such that 
$\sum_{i\in I}\ \alpha_i(m')m_i=$ $m'$ for all $m'\in$ $M$.
Similarly, we have a natural isomorphism
\begin{Equation} \label{MveeMadj}
$$\Hom_B(M^\vee\tenA U,V) \cong \Hom_A(U,M\tenB V)$$
\end{Equation}
obtained from \ref{MMveeadj} by
exchanging the roles of $A$ and $B$ and using $M^\vee$
instead of $M$ together with the canonical double duality 
$M^{\vee\vee}\cong$ $M$. The adjunction unit and counit
of this adjunction are  represented by bimodule homomorphisms
\begin{Equation} \label{MveeMadjunit}
$$\epsilon_{M^\vee} : A \longrightarrow M\tenB M^\vee\ ,\ \ 
1_A \mapsto \sum_{j\in J}\ m_j\ten (t\circ \beta_j)\ ,$$
$$\eta_{M^\vee} : M^\vee\tenA M\longrightarrow B\ ,\ \ (t\circ\beta)\ten m \mapsto
\beta(m)\ ,$$
\end{Equation}
where $J$ is a finite indexing set, $\beta_j\in$ $\Hom_{B^\op}(M,B)$,
$m_j\in$ $M$, such that $\sum_{j\in J}\ m_j\beta_j(m')=$ $m'$ for
all $m'\in$ $M$, where $m\in$ $M$ and $\beta\in$ $\Hom_{B^\op}(M,B)$.
Note that $\eta_M\circ \epsilon_{M^\vee}$ is an $A$-$A$-bimodule endomorphism
of $A$, hence given by left or right multiplication with an element in $Z(A)$.
Similarly, $\eta_{M^\vee}\circ \epsilon_{M}$ is a $B$-$B$-bimodule endomorphism
of $B$, hence given by left or right multiplication with an element in $Z(B)$.
Following \cite[Definition 3.1]{Lintransfer}, we set
\begin{Equation} \label{rel-proj-element}
$$\pi_M = (\eta_M\circ \epsilon_{M^\vee})(1_A)$$
$$\pi_{M^\vee} = (\eta_{M^\vee}\circ \epsilon_M)(1_B)$$
\end{Equation}
We call $\pi_M$ the {\it relatively $M$-projective
element of} $Z(A)$. Similarly, $\pi_{M^\vee}$ is called the
{relatively $M^\vee$-projective element of} $Z(B)$. These elements depend on the choices
of the symmetrising forms of $A$ and $B$; see \cite[Remark 3.2]{Lintransfer} for
details. 

\begin{Remark} \label{ring-extend-adj}
The adjunction isomorphisms \ref{MMveeadj}, \ref{MveeMadj},
 and the associated adjunction
units and counits in the equations \ref{MMveeadjunit}, \ref{MveeMadjunit}
commute with extensions of the ring of scalars $R$, where we use the fact that
$M$, $M^\vee$ are finitely generated projective as left and right modules. 
More precisely, if $R\to S$ is
a homomorphism of commutative rings through which $S$ is regarded as an $R$-module,
then, writing $SU=S\tenR U$ and $SU^\vee=$ $\Hom_S(SU,S)$ for any $R$-module $U$, 
we have a canonical isomorphism $S(M^\vee\tenB M)\cong$ $(SM^\vee\ten_{SB} SM$ through which $\Id_S\ten\epsilon_M $ becomes the adjunction unit of $SM\ten_{SB}-$ being
left adjoint to $SM^\vee\ten_{SA}-$. Similar statements hold for the remaining adjunction
unit and the counits. This will be needed in the proofs of the two main theorems for
the extension from a complete discrete valuation ring to its field of fractions.
\end{Remark}

\begin{Remark} \label{adj-additive}
The adjunction isomorphism \ref{MMveeadj} is additive in $M$.  Thus the
adjunction unit and counit in \ref{MMveeadjunit} are additive in $M$ in
the following sense: given two perfect $A$-$B$-bimodules $M$, $N$, the
adjunction unit
$$\epsilon_{M\oplus N} : B \to (M\oplus N)^\vee \tenA (M\oplus N)$$
is equal to the omposition of
$$\epsilon_M + \epsilon_N : B\to M^\vee\tenA M \oplus N^\vee \tenA  N$$
followed by the canonical inclusion of the right side into
$(M\oplus N)^\vee \tenA (M\oplus N)$. Similarly, the adjunction counit
$$\eta_{M\oplus N} : (M\oplus N)\tenB (M\oplus N)^\vee \to A$$
is equal to the map 
$$\eta_M+\eta_N : M\tenA M^\vee \oplus N\tenA N^\vee \to A$$
extended by zero on the mixed summands $M\tenA N^\vee$ and
$N\tenA M^\vee$.  The analogous statements hold for the adjunction
isomorphism \ref{MveeMadj} and the corresponding adjunction unit
and counit in \ref{MveeMadjunit}.
\end{Remark}

\begin{Remark}
If $U$, $V$ have in addition right $C$-module structures for some further $R$-algebra
$C$, then the isomorphisms in \ref{MAMvee}, \ref{RAdualten} are isomorphisms
of right $C$-modules. Thus the isomorphism \ref{MMveeadj} induces an isomorphism
$$\Hom_{A\tenR C^\op}(M\tenB V,U) \cong \Hom_{B\tenR C^\op}(V,M^\vee\tenA U) .$$
\end{Remark}

\section{Transfer for symmetric algebras} \label{transferSection}

Let $R$ be a commutative Noetherian ring (with unit element), and let $A$, $B$
be symmetric $R$-algebras, with symmetrising forms $s$, $t$, respectively.
Let $M$ be a perfect $A$-$B$-bimodule. 
Following \cite{Broue1}, for finitely generated $A$-modules $U$, $V$, we have a 
transfer map 
\begin{Equation} \label{trM-notation}
$$\tr_M = \tr_M(U,V)  : \Hom_B(M^\vee\tenA U, M^\vee \tenA V) \to \Hom_A(U,V)$$
\end{Equation}
sending a $B$-homomorphism $\beta : M^\vee\tenA U\to M^\vee\tenA V$ to the
$A$-homomorphism  
\begin{Equation} \label{trM-Def}
$$\tr_M(\beta) = (\eta_M\ten\Id_V) \circ (\Id_M\ten \beta) \circ 
(\epsilon_{M^\vee} \ten \Id_U).$$ 
\end{Equation}
  More explicitly, $\tr_M(\beta)$ is  the composition of   $A$-homomorphisms
$$\xymatrix{U \ar[rr]^(.35){\epsilon_{M^\vee}\ten \Id_U} & & M\tenB M^\vee \tenA U 
\ar[rr]^{\Id_{M}\ten \beta} & & M\tenB M^\vee \tenA V \ar[rr]^(.59){\eta_M\ten \Id_V} & & V}$$
with the standard identifications $A\tenA U=U$ and $A\tenA V=V$. 
The functors $M\tenB-$ and $M^\vee\tenA-$ are exact and 
preserve finitely generated projective modules. Therefore, if $\beta$ factorises
through a projective $B$-module, then $\tr_M(\beta)$ factorises through a
projective $A$-module, and hence  $\tr_M$ induces a well-defined map, still denoted
\begin{Equation} \label{tr-M-stable}
$$\tr_M  : \Hombar_B(M^\vee\tenA U, M^\vee \tenA V) \to \Hombar_A(U,V).$$
\end{Equation}
It also follows that in the stable module category $\modbar(A)$, 
for any integer $n$, we  have unique isomorphisms
$$\Sigma_A^n(M\tenA M^\vee \tenA U) = M \tenA \Sigma_B^n(M^\vee\tenA U) =
M\tenB M^\vee \tenA \Sigma_A^n(U)$$
and through these identifications and their analogues, we have an equality
of morphisms in the stable category $\modbar(A)$
\begin{Equation} \label{transferSigmacommute}
$$\Sigma^n_A(\tr_M(\beta)) = \tr_M(\Sigma_B^n(\beta)) : \Sigma_A^n(U) \to 
\Sigma_A^n(V).$$
\end{Equation}
By  \cite[\S 7.1]{Ligrblock}
there are graded versions of these transfer maps for Tate and Tate-Hochschild
cohomology.
An element in
$\hatExt^n_B(M^\vee\tenA U,M^\vee\tenA V)$ is represented by a $B$-homomorphism
$\beta : M^\vee\tenA U\to$ $M^\vee\tenA \Sigma^n(V)$, where we identify
$\Sigma^n(M^\vee\tenA U)=$ $M^\vee\tenA \Sigma^n(V)$ and where we use the same letter
$\Sigma$ for either $\Sigma_A$ or $\Sigma_B$. The transfer map
$\tr_{M}$ sends $\beta$ to the element
$\tr_M(\beta)$ in $\Ext_A^n(U, V)$ represented by the
$A$-homomorphism, abusively also denoted $\tr_M(\beta)$, given by
\begin{Equation} \label{trM-graded-notation}
$$\tr_M(\beta) = (\eta_M\ten \Id_{\Sigma^n(V)})\circ (\Id_M\ten \beta) 
\circ (\epsilon_{M^\vee} \ten \Id_U)$$
\end{Equation}
with the standard identifications $A\tenA U=U$ and  $A\tenA \Sigma^n(V)=\Sigma^n(V)$.
 More explicitly, $\tr_M(\beta)$ is obtained as the composition
\begin{Equation}
$$\xymatrix{U \ar[rr]^(.35){\epsilon_{M^\vee}\ten\Id_U} & & M\tenB M^\vee \tenA U 
\ar[rr]^(.45){\Id_{M}\ten\beta} & & M\tenB M^\vee \tenA \Sigma^n(V)
\ar[rr]^(.61){\eta_{M}\ten\Id_{\Sigma^n(V)}} & &  \Sigma^n(V) }$$
\end{Equation}

\medskip
A variation of the same principle applied to bimodules yields in particular a transfer
for Tate-Hochschild cohomology.
We use again simply  $\Sigma$ instead of $\Sigma_{A\tenk A^{\op}}$ or
$\Sigma_{B\tenk B^{\op}}$. An element $\zeta\in$ $\hatHH^n(B)$
is represented by a $B$-$B$-bimodule homomorphism, abusively denoted
by  the same letter,  $\zeta : B\to$ $\Sigma^n(B)$. We denote
by $\tr_M(\zeta)$ the element in $\hatHH^n(A)$ represented
by the $A$-$A$-bimodule homomorphism
$$\xymatrix{
M\tenB M^\vee = M\tenB B\tenB M^\vee
\ar[rrr]^(.45){\Id_M\ten\zeta\ten\Id_{M^\vee}} & & &
M\tenB \Sigma^n(B)\tenB M^\vee=\Sigma^n(M\tenB M^\vee) }$$
precomposed with the adjunction unit $\epsilon_{M^\vee}: A\to$
$M\tenB M^\vee$ and composed with the `shifted' adjunction counit
$\Sigma^n(\eta_M) : \Sigma^n(M\tenB M^\vee)\to$ $\Sigma^n(A)$.
The identification $M\tenB \Sigma^n(B)\tenB M^\vee=$
$\Sigma^n(M\tenB M^\vee)$ is to be understood as the
canonical isomorphism in $\modbar(A\tenk A^{\op})$, using the
fact that the functor $M\tenB - \tenB M^\vee$ sends a
projective resolution of the $B$-$B$-bimodule $B$ to
a projective resolution of the $A$-$A$-bimodule $M\tenB M^\vee$.
Modulo this identification, we thus have graded $k$-linear map
\begin{Equation} \label{trM-HHstar}
$$\tr_M : \hatHH^*(B) \longrightarrow \hatHH^*(A)$$
\end{Equation}
defined by 
\begin{Equation} \label{trM-HH}
$$\tr_M(\zeta) = \Sigma^n(\eta_M)\circ
(\Id_M\ten\zeta\ten\Id_{M^\vee})\circ \epsilon_{M^\vee}\ .$$
\end{Equation}
Note that $\tr_M$ is not necessarily a
multiplicative map from $\hatHH^*(B)$ to $\hatHH^*(A)$. 
In all the cases above, we have analogous transfer maps $\tr_{M^\vee}$
obtained from exchanging the roles of $A$ and $B$. Using $\Ext$ instead
of $\hatExt$ yields the transfer maps introduced in \cite{Lintransfer}.
The two are well-known to coincide for $n>0$. 

Suppose that $A$ is $R$-free. Let $X$ be an $R$-basis of $A$ and $X^\vee$ the dual basis with
respect to the symmetrising form $s$ on $A$; that is, we have a bijection $x\mapsto x^\vee$
from $X$ to $X^\vee$ such that $s(xx^\vee)=1$ for $x\in X$ and $s(xy^\vee)=0$ for 
$x$, $y \in X$ such that $x\neq y$. The element
\begin{Equation} \label{relprojelement}
$$ z_A = \sum_{x\in X} x x^\vee$$
\end{Equation}
is called the {\em relative projective element with respect to $s$}. One easily checks
that this is an element in $Z(A)$ which does not depend on the choice of the basis $X$, but 
which does depend on the choice of $s$. If $s'$ is another symmetrising form, then
there is a unique element $z\in Z(A)^\times$ such that $s'(a) = s(za)$ for all $a\in A$.
If $X^\vee$ is as before the dual basis of $X$ with respect to $s$, then $z^{-1}X^\vee$
is the dual basis of $X$ with respect to $s'$, and hence the relative projective 
element with respect to $s'$ is equal to $z'_A = z^{-1}z_A$. 

\begin{Remark}
If we regard $R$ as a
symmetric algebra with the identity map as symmetrising form and take for $M$
the $A$-$R$-bimodule $A$ (that is, the regular bimodule $A$ restricted to $R$ on the right),
then $z_A$ is the relative $M$-projective element $\pi_M$ defined in
\ref{rel-proj-element} above.  That is, $z_A$ is the image of $1_A$ under the composition
of bimodule homomorphisms
$A \to A\tenR A \to A$, where the second map is given by multiplication in $A$
and the first map is obtained by dualising the multiplication map and then using the
isomorphism $A^\vee\cong A$ and $(A\tenR A)^\vee\cong$ $ A^\vee\tenR A^\vee\cong$
$A\tenR A$. This definition of $z_A$ has the advantage of not needing
$A$ to be free over $R$ but just finitely generated projective as an $R$-module.
For the purpose of this paper we do not need this generality.
\end{Remark}

 Tate duality for Tate-Hochschild cohomology involves
bimodules, and hence we will need the following well-known description of 
relative projective elements for tensor products of symmetric algebras as well as 
their compatibility with the passage to blocks.

\begin{Lemma} \label{zAtenB} 
Let $A$, $B$ be $R$-free symmetric $R$-algebras, with symmetrising forms $s$, $t$, respectively.
Then $A^\op$ is symmetric algebra with $s$ as  symmetrising form, $A\tenR B$ is symmetric
with $s\ten t$ as symmetrising form, and $A\times B$ is symmetric with symemtrising
form $s+t$. With respect to these symmetrising forms, we have
\begin{itemize}
\item[{\rm (i)}]
$z_{A^\op} = z_A$.
\item[{\rm (ii)}] 
$z_{A\tenB B} = z_A \ten z_B$.
\item[{\rm (iii)}]
$z_{A\times B} = (z_A, z_B)$.
\end{itemize}
\end{Lemma} 

\begin{proof}
A trivial verification shows that $s$, $s\ten t$,  and $s+t$ are symmetrising forms of 
$A^\op$, $A\tenR B$, and $A\times B$, respectively. 
Let $X$ be an $R$-basis of $A$, with dual basis $X'$ and corresponding bijection 
$x\mapsto x'$ fron $X$ to $X'$ as in \ref{relprojelement} above. Then $X$ and $X'$ 
are also dual to each other with 
respect to $s$ as a symmetrising form of $A^\op$. This implies $z_{A^\op}=z_A$, whence (i).
Let $Y$ be an $R$-basis of $B$ with dual basis $Y'$ and corresponding bijection $y\mapsto y'$
for $y\in Y$. Then the image in $A\tenR B$ of $X\ten Y$ is an $R$-basis, and its dual basis with
respect to $s\ten t$ is $X'\ten Y'$, with the bijection from $X\ten Y$ to $X'\ten Y'$ 
mapping $x\ten y$ to $x'\ten y'$, for $x\in X$ and $y\in Y$. It follows that
$$z_{A\tenR B} = \sum_{x\in X, y\in Y}\ xx' \ten yy' = (\sum_{x\in X} xx') \ten (\sum_{y\in Y} yy') =
z_A \ten z_B$$
as stated in (ii). The union $(X \times \{0\}) \cup (\{0\}\times Y)$ is an $R$-basis of 
$A\times B$ with dual
basis $(X'\times \{0\}) \cup (\{0\}\times Y')$. Statement (iii) follows.
\end{proof} 

\begin{Remark} \label{transfer-additive}
The additivity properties of adjunction units and counits mentioned in 
Remark \ref{adj-additive} as well as the additivity of shift functors  on
stable module categories  imply that the transfer maps above are
additive in $M$. More precisely, for $M$, $N$ perfect $A$-$B$-bimodules,
we have $\tr_{M\oplus N} = \tr_M + \tr_N$
for all the variations of transfer maps $\tr_M$ considered in
\ref{trM-notation}, \ref{trM-graded-notation}, \ref{trM-HH} above.
\end{Remark}

\begin{Remark} \label{notationRemark}
The transfer map in Tate-Hochschild cohomology  \ref{trM-HHstar}, \ref{trM-HH}
is not strictly speaking a special case
of the transfer maps $\tr_M(U,V)$, but the two are related via a generalisation
of $\tr_M(U,V)$. Let $A$, $B$, $C$ be symmetric $R$-algebras, and let $U$,$V$
be finitely generated $A\tenR C^\op$-modules.  The transfer map $\tr_M=$
$\tr_M(U,V)$ from \ref{trM-notation}  induces a map, yet again denoted
$$\tr_M = \tr_M(U,V)  : \Hom_{B\tenR C^\op}(M^\vee\tenA U, M^\vee \tenA V) \to 
\Hom_{A\tenR C^\op}(U,V)$$
sending a $B\tenR C^\op$-homomorphism $\beta : M^\vee\tenA U\to M^\vee\tenA V$ 
to the $A\tenR C^\op$-homomorphism  
$$\tr_M(\beta) = (\eta_M\ten\Id_V) \circ (\Id_M\ten \beta) \circ 
(\epsilon_{M^\vee} \ten \Id_U).$$ 
The functor $M\tenB-$ sends a projective $B\tenR C^\op$-module to a projective
$A\tenR C^\op$-module, and hence if $\beta$ factors through a projective
$B\tenR C^\op$-module, then $\Id_M\ten\beta$ factors through a projective
$A\tenR C^\op$-module. Thus $\tr_M$ induces a well-defined map
$$\tr_M  : \Hombar_{B\tenR C^\op}(M^\vee\tenA U, M^\vee \tenA V) \to  
\Hombar_{A\tenR C^\op}(U,V).$$
Applied with $C=A$, $U=A$, $V=\Sigma_{A^e}^n(A)$, this yields a map
\begin{Equation} \label{trM-HHn}
$$\tr_M : \Hombar_{B\tenR A^\op}(M^\vee, \Sigma^n(M^\vee)) \to
\Hombar_{A^e}(A,\Sigma^n(A))$$
\end{Equation}
where we have made use of the standard identifications
$\Sigma^n(B)\tenB M^\vee\cong$ $\Sigma^n(M^\vee)\cong$ $M^\vee\tenA \Sigma^n(A)$
 in the stable category $\modbar(B\tenR A^\op)$.
The functor $-\tenB M^\vee$ induces a graded algebra homomorphism
$\hatHH^*(B) =$ $\hatExt^*_{B^e}(B,B) \to$ 
$\hatExt^*_{B\tenR A^\op}(M^\vee, M^\vee)$, and composing this
with the map $\tr_M$ from \ref{trM-HHn} yields the transfer map in Hochschild
cohomology $\hatHH^*(B) \to$ $\hatHH^*(A)$ from \ref{trM-HHstar}, \ref{trM-HH}.
\end{Remark}

\section{Adjunction maps for matrix algebras}

We need to identify the adjunction maps and the transfer maps reviewed in
the previous section in the case that $A$, $B$ are matrix
algebras. This is elementary linear algebra, so we just give some
pointers towards verifications.  Let $R$ be a commutative ring.
Let  $U$, $V$ be free $R$-modules of finite ranks over $R$.
Set $A=\End_R(U)$ and $B=\End_R(V)$. Then $A$ and $B$ are symmetric
$R$-algebras
with symmetrising forms the trace maps $\trace_U$, $\trace_V$, sending a
$R$-linear endomorphism of $U$, $V$ to its trace, respectively. Any other
symmetrising form of $A$, $B$ is of the form $\rho\cdot\trace_U$, $\rho\cdot\trace_V$
for some $\rho \in R^\times$, respectively. Set $M=U\tenR V^\vee$.
Tensoring with $M$ and its dual is the simplest instance of a Morita equivalence;
all we need to make sure in this Section is that the standard maps in this context 
are indeed the adjunction maps with respect to the trace maps as symmetrising 
forms. These verifications make use of the following well-known Lemma which
links traces to adjunction maps.

\begin{Lemma} \label{matrix-adj}
We have an isomorphism $\Hom_A(U,A)\cong U^\vee$ sending $\lambda\in$
$\Hom_A(U,A)$ to $\trace_U\circ \lambda$. 
We have a commutative diagram of $A$-$A$-bimodule homomorphisms
$$\xymatrix{ & & U\tenR \Hom_A(U,A) \ar[lld]_\sigma \ar[drr]^\rho& & \\
U\tenR U^\vee \ar[rrrr]^\alpha \ar[rrd]_{\tau} & & && A \ar[lld]^{\trace_U} \\
 & & R & & }$$
where $\alpha$ sends $u\ten \mu$ to the endomorphism
$u'\mapsto \mu(u')\ten u$, $\sigma$ sends
$u\ten \lambda$ to $u\ten (\trace_U\circ\lambda)$,
 $\rho$ sends $u\ten \lambda$ to $\lambda(u)$, and 
 $\tau$ sends
$u\ten \mu$ to $\mu(u)$, for all $u$, $u'\in U$, $\lambda\in\Hom_A(U,A)$, and
$\mu\in U^\vee$.  
\end{Lemma}

\begin{proof}
The first statement is a special case of the isomorphism \ref{MAMvee}.
The commutativity of the lower triangle is well-known; see for instance
 \cite[Proposition 2.10.2]{LiBookI} for a proof. The commutativity of the
upper diagram is an easy verification.
\end{proof}

We identify $M^\vee$ with $V\tenR U^\vee$ via the obvious isomorphisms 
$$M^\vee= (U\tenR V^\vee)^\vee= V^{\vee\vee}\tenR U^\vee=
V\tenR U^\vee. $$ 
This leads to  identifications
$$M\tenB M^\vee = U\tenR V^\vee \tenB V \tenR U^\vee = U\tenR U^\vee$$
where we identify $V^\vee\tenB V = R$ via the map $\nu \ten v \to \nu(v)$,
for $v\in V$ and $\nu \in V^\vee$.  Similarly, we identify
$M^\vee\tenA M = V\tenR V^\vee$. 
Let $\CB$ be an $R$-basis of $U$, with dual basis in $U^\vee$
denoted $\CB^\vee$. For $u\in \CB$
we denote by $u^\vee$ the unique element in $\CB^\vee$ satisfying
$u^\vee(u)=1$ and $u^\vee(u')=0$ for $u'\in$ $\CB$, $u'\neq u$. Similarly,
let $\CC$ be an $R$-basis of $V$, with dual basis in $V^\vee$ 
denoted analogously $\CC^\vee$.
The adjunction units and counits from the preceding section in this case 
(with the choice of symmetrising forms $\trace_U$, $\trace_V$) are all isomorphisms, and
their precise descriptions are as follows.

\begin{Equation} \label{MMveeadjunitM}
$$\epsilon_M : B \longrightarrow M^\vee\tenA M = V\tenR V^\vee \ ,\ \ 
1_B \mapsto \sum_{v\in\CC}\  v \ten v^\vee,$$
$$\eta_M : U\tenR U^\vee = M\tenB M^\vee\longrightarrow A\ ,\ \ u \ten \mu\mapsto
(y \mapsto \mu(y)u)\ ,$$
\end{Equation}
where $u$, $y\in U$, $\mu\in U^\vee$. 
\begin{Equation} \label{MveeMadjunitM}
$$\epsilon_{M^\vee} : A \longrightarrow M\tenB M^\vee = U\tenR U^\vee\ ,\ \ 
1_A \mapsto \sum_{u\in \CB}\  u \ten u^\vee\ ,$$
$$\eta_{M^\vee} : V\tenR V^\vee = M^\vee\tenA M\longrightarrow B\ ,\ \ v\ten \nu \mapsto
(w \mapsto \nu(w)v)\ ,$$
\end{Equation}
where $v$, $w\in V$ and $\nu \in V^\vee$. 
We further note that
\begin{Equation} \label{MMvee-adj-inv}
$$\epsilon_{M^\vee} = (\eta_M)^{-1},\ \ \ \ \ \eta_{M^\vee} = (\epsilon_M)^{-1}.$$
\end{Equation}
An easy verification shows  that the relative projective elements in $Z(A)$ and $Z(B)$
with respect to  the symmetrising forms $\trace_U$ and $\trace_V$, respectively, 
are  equal to 
\begin{Equation} \label{zAzB-matrix}
$$z_A = \rk_R(U)\cdot 1_R, \ \ \ \ \ \ \ \ z_B = \rk_R(V)\cdot 1_R.$$
\end{Equation}

\begin{Remark} \label{symm-form-independent}
Let $s$, $t$ be symmetrising forms of $A$, $B$, respectively. Then 
$s = \lambda \cdot\trace_U$ and $t = \mu \cdot \trace_V$ for some $\lambda$, 
$\mu \in R^\times$.
Denoting by $\epsilon'_M$, $\eta'_M$, $\epsilon'_{M^\vee}$, $\eta'_{M^\vee}$
the adjunction maps from \ref{MMveeadjunitM} and \ref{MveeMadjunitM} with
respect to $s$, $t$, it follows that
$$\epsilon'_M = \lambda\epsilon_M,\ \ \ \ \ 
\eta'_M = \lambda^{-1} \eta_M, \ \ \ \ \ \epsilon_{M^\vee} = \mu \epsilon_{M^\vee},
\ \ \ \ \ \eta'_{M^\vee} = \mu^{-1} \eta_{M^\vee}.$$
Thus the trace map $\tr'_M  : \End_B(M^\vee\tenA U) \to \End_A(U)$ 
with respect to $s$ and $t$ satisfies
$$\tr'_M = \lambda^{-1} \mu \tr_M.$$
The relative projective central elements $z'_A$ and $z'_B$ with respect to $s$, $t$
are 
$$z'_A = \lambda^{-1} z_A = \lambda^{-1}\rk_R(U), \ \ \ \ \ z'_B = \mu^{-1} z_B=
\mu^{-1} \rk_R(V).$$
\end{Remark}

\section{Transfer for matrix algebras} 

Let $K$ be a field of characteristic zero, and let  $U$, $V$ be
finite-dimensional $K$-vector spaces. We set $A=\End_K(U)$ and $B=\End_K(V)$, regarded
as symmetric algebras with symmetrising forms $\trace_U$ and $\trace_V$, respectively.
We set $M = U\tenK V^\vee$. We note that since $A\tenK A^\op$ and $A\tenK B^\op$ are simple
algebras, it follows that every finitely generated $A$-$A$-bimodule is projective and isomorphic
to a finite direct sum of copies of $A$, and every finitely generated $A$-$B$-bimodule is projective
and  isomorphic to a finite  direct sum of copies of the simple $A$-$B$-bimodule $M$.  
For finitely generated $A$-modules
$U'$, $U''$ we denote by 
$$\varphi_A : \Hom_A(U',U'') \times \Hom_A(U'',U') \to K$$
the bilinear map  sending $(\alpha, \beta)$  to $z_A^{-1} \trace_{U'}(\beta\circ \alpha)$. 
We use the analogous
notation  $\varphi_B$ for finitely generated $B$-modules.
We keep the above notation throughout this section. 

\begin{Proposition} 
Let $U'$, $U''$ be finitely generated $A$-modules.
Then the bilinear map
$$\varphi_A : \Hom_A(U',U'') \times \Hom_A(U'',U') \to K$$
is non-degenerate.
\end{Proposition}

\begin{proof} 
This is a trivial consequence of \cite[Proposition 2.1]{EGKL}, and easily checked directly.
\end{proof}

The following result is needed for the proof of Theorem \ref{Thm1}.

\begin{Proposition} \label{Tate-matrix-1}
Let $U'$, $U''$ be finite-dimensional $A$-modules.
 Let $\alpha : U'\to U''$ be an $A$-homomorphism
and let $\beta : M^\vee\tenA U' \to M^\vee\tenA U''$ be a $B$-homomorphism. 
For any choice of symmetrising forms on $A$ and $B$ we have 
$$\varphi_A(\alpha, \tr_M(\beta)) = \varphi_B(\Id_{M^\vee}\ten \alpha, \beta). $$
\end{Proposition}

\begin{proof}
Assume first that the symmetrising forms are $\trace_U$ and $\trace_V$.
Note that $U'$, $U''$ are isomorphic to finite direct sums of copies of $U$. Since
$\varphi_A$, $\varphi_B$ are additive in both components, we may assume
that $U'=U''=U$. Then $\alpha$ is a $K$-linear multiple, and since both
sides are bilinear, we may assume that $\alpha=\Id_U$. Thus stated equation
is equalent  to 
$$z_A^{-1}\trace_U(\tr_M(\beta)) = z_B^{-1} \trace_{M^\vee\tenA U}(\beta)$$
Note that $M^\vee\tenA U\cong V$, and hence $\beta=\lambda \Id_{M^\vee\tenA U}$ for
some $\lambda \in K$.  In particular, we have
$$\trace_{M^\vee\tenA U}(\beta) = \lambda \dim_K(V).$$
By \ref{trM-Def}, we have
$\tr_M(\beta) = $ $\eta_M \circ (\Id_M\ten \beta) \circ \epsilon_{M^\vee} =$
$\lambda \eta_M\circ \epsilon_{M^\vee}=$ $\lambda \Id_U$, and hence
$$\trace_U(\tr_M(\beta)) = \dim_K(U)\cdot 1_K .$$
Using $z_A=\dim_K(U)\cdot 1_K$ and $z_B=\dim_K(V)\cdot 1_K$ the
result follows for the chosen symmetrising forms $\trace_U$ and $\trace_V$.
The result for arbitrary symmetrising forms follows easily from the
Remark \ref{symm-form-independent}.
\end{proof}

The next result, which is a variation of the previous Proposition, will be needed in
the proof of Theorem \ref{Thm2}.  Set $A^e = A\tenK A^\op$.
By Lemma \ref{zAtenB}, we have $z_{A^e} =$
$z_A\ten z_A$. Thus the inverse of this element acts on an $A$-$A$-bimodule by simultaneously
multiplying by $z_A^{-1}$ on the left and on the right. In particular, this element acts
on $A$  by multiplication by $z_A^{-2}$.

\begin{Proposition} \label{Tate-matrix-2}
Let $X$ be a finitely generated $A$-$A$-bimodule and let $Y$ be a finitely generated
$B$-$B$-bimodule.  Let $\zeta : A\to X$ be an $A$-$A$-bimodule homomorphism,
let $\xi : X\tenA M \to M\tenB Y$ be an $A$-$B$-bimodule
homomorphism, and let $\sigma : Y \to B$ be a $B$-$B$-bimodule homomorphism. 
For any choice of symmetrising forms on $A$ and on $B$, the
trace on $A$ of the map
$$z_A^{-2}\cdot \eta_M\circ (\Id_M\ten\sigma\ten\Id_{M^\vee}) \circ (\xi \ten\Id_{M^\vee})
\circ (\Id_X \ten \epsilon_{M^\vee}) \circ \zeta$$
is equal to the trace on $B$ of the map
$$z_B^{-2}\cdot \sigma \circ \eta_{M^\vee} \circ (\Id_{M^\vee}\ten\xi) \circ
(\Id_{M^\vee}\ten\zeta \ten\Id_M) \circ \epsilon_M.$$
\end{Proposition}

\begin{proof}
Both maps in the statement are additive in $X$. Since $A$ is up to isomorphism the
unique indecomposable $A$-$A$-bimodule, it follows that $X$ is isomorphic to a 
finite direct sum of copies of $A$, and hence we may assume that $X=A$.
For the same reason we may assume that $Y=B$. Then $\xi$ becomes an
$A$-$B$-bimodule endomorphism of the simple $A$-$B$-bimodule $M=U\tenK V^\vee$,
hence is equal to multiplication by a scalar, which we will denote abusively again by $\xi$.
Similarly, $\sigma$ becomes a bimodule endomorphism of $B$, so is given by
multiplication with a scalar, again denoted by $\sigma$, and $\zeta$ becomes
a bimodule endomorphism of $A$, given by multiplication with a scalar, again
denoted by $\zeta$. Thus the two maps in the statement take the form
\begin{Equation} \label{e1}
$$z_A^{-2}\sigma\xi\zeta \cdot (\eta_M \circ \epsilon_{M^\vee}), $$
$$z_B^{-2}\sigma\xi\zeta \cdot (\eta_{M^\vee}\circ\epsilon_M).$$
\end{Equation}
Let  $s$, $t$ be symmetrising forms of $A$, $B$. Then $s = \lambda \cdot\trace_U$
and $t = \mu \cdot\trace_V$ for some $\lambda$, $\mu\in K^\times$. 
It follows from the Remark \ref{symm-form-independent} and \ref{MMvee-adj-inv}
that then the adjunction units and counits with respect to these symmetrising
forms satisfy 
$$\eta_M\circ \epsilon_{M^\vee} = \lambda^{-1}\mu\Id_A, $$
$$\eta_{M^\vee}\circ\epsilon_M = \lambda\mu^{-1}\Id_B. $$ 
Also by the Remark \ref{symm-form-independent}, the
relative projective elements are $z_A = \lambda^{-1} \dim_K(U)$ and
$z_B = \mu^{-1}\dim_K(V)$.  Thus $z_A^{-2}=$ $\lambda^2\dim_K(U)^{-2}$ and
$z_B^{-2}=$ $\mu^2\dim_K(V)^{-2}$. 
Therefore the two maps in \ref{e1} are equal to  the two maps
$$\dim_K(U)^{-2}\lambda\mu\sigma\xi\zeta \cdot \Id_A, $$
$$\dim_K(V)^{-2}\lambda\mu\sigma\xi\zeta\cdot \Id_B.$$
Since $\dim_K(U)^2=\dim_K(A)$ and $\dim_K(V)^2=\dim_K(B)$ it follows
that both maps have the same trace, equal to $\lambda\mu\sigma\xi\zeta$. 
\end{proof}

\section{Proof of Theorem \ref{Thm1}}

Let $\CO$ be a complete discrete valuation ring with field of fractions $K$ of 
characteristic zero. Let $A$, $B$ be symmetric $\CO$-algebras such that
$K\tenO A$ and $K\tenO B$ are semisimple. Note that then $K\tenO A$ and
$K\tenO B$ are separable since $\chr(K)=0$. Fix symmetrising forms $s$, $t$
of $A$, $B$, respectively. Let $M$ be an $A$-$B$-bimodule which is 
finitely generated projective as a left $A$-module and as a right $B$-module.
Let $U$, $V$ be finitely generated $\CO$-free $A$-modules.
We write $KA$ instead of $K\tenO A$ and $KU$ instead of $K\tenO U$; similarly
for $B$ and $V$. We identify $\Hom_{KA}(KU, KV)=$ $K\Hom_A(U,V)$ whenever
convenient, and we identify $\Hom_A(U,V)$ with its image in this space.

In degree zero,  we have $\hatExt^0_A(U,V)=$ $\Hombar_A(U,V)$, and
Tate duality takes the following form. 
By \cite[Proposition 2.1]{EGKL} we have a non-degenerate bilinear form
\begin{Equation} \label{Tatedegree0-1}
$$\varphi_{KA}(-,-) : K\Hom_A(U,V) \times K\Hom_A(V,U) \to K$$
\end{Equation}
which sends $(\alpha, \beta)\in$ $K\Hom_A(U,V)\times K\Hom_A(V,U)$ to
the trace   on $KU$ of the 
$KA$-endomorphism $z_A^{-1}\beta\circ\alpha$ of $KU$.
This restricts to an $\CO$-bilinear form
\begin{Equation} \label{Tatedegree0-2}
$$\varphi_A : \Hom_A(U,V) \times \Hom_A(V,U) \to K$$
\end{Equation}
By \cite[Theorem 1.3]{EGKL} and its proof in \cite[\S 2]{EGKL}, this form sends
$\Hom_A^\pr(U,V)\times\Hom_A(V,U)$ and $\Hom_A(U,V)\times \Hom^\pr_A(V,U)$
to $\CO$, and the induced bilinear  form
\begin{Equation} \label{Tatedegree0-3}
$$\langle - , - \rangle_A : \Hombar_A(U,V) \times \Hombar_A(V,U) \to K/\CO$$
\end{Equation}
is non-degenerate. We note that for any $\alpha\in\Hom_A(U,V)$ and
$\beta\in\Hom_A(V,U)$ we have 
\begin{Equation} \label{Tate-symmetric}
$$\varphi_{KA}(\alpha, \beta) = \varphi_{KA}(\beta,\alpha),$$
$$ \langle \alpha, \beta \rangle_A = \langle \beta , \alpha \rangle_A.$$
\end{Equation}
To see this, observe that left multiplication by $z_A^{-1}$ commutes with all 
$KA$-homomorphisms. Thus
the $KA$-endomorphism $z_A^{-1} (\beta \circ \alpha)$ of $KU$ is equal to
$(z_A^{-1}\beta) \circ \alpha$, hence has the same trace on $KU$ as the
endomorphism $\alpha\circ (z_A^{-1}\beta)$ on $KV$. The latter is equal to
$z_A^{-1}(\alpha\circ \beta)$.

\begin{proof}{Proof of Theorem \ref{Thm1}}
We start by proving Theorem \ref{Thm1} in degree zero. 
Tate duality in degree zero takes the form as reviewed in \ref{Tatedegree0-1},
\ref{Tatedegree0-2}, \ref{Tatedegree0-3}  just above. 
Since the functors  $M\tenB-$ and $M^\vee\tenA-$
preserve finitely generated projective modules over $A$ and $B$, it follows that
if $\beta\in\Hom_A^\pr(M^\vee\tenA U, M^\vee\tenA V)$, then
$\tr_M(\beta) \in \Hom_A^\pr(U,V)$. Similarly, if $\alpha\in\Hom_A^\pr(U,V)$, then
$\Id_{M^\vee} \ten \alpha\in$ $\Hom_B^\pr(M^\vee\tenA U, M^\vee\tenA V)$.
That is, it suffices to show the equality
$$\varphi_{KA}(\alpha, \tr_M(\beta)) = \varphi_{KB}(\Id_{KM^\vee}\ten\alpha, \beta)$$
where $\alpha\in K\Hom_A(U,V)$ and $\beta \in$ 
$ K\Hom_B(M^\vee\tenA U, M^\vee\tenA V)$. This equation holds if and only if it
holds for field extensions of $K$, so we may assume that $KA$, $KB$ are split
semisimple. That is, $KA$, $KB$ are direct products of matrix algebras. 
Since both sides are additive, we may in fact assume that $KA$, $KB$ are
matrix algebras. In that case, the equation follows from Proposition
\ref{Tate-matrix-1}. Together with \ref{Tate-symmetric}, 
this proves Theorem \ref{Thm1} for $n=0$. 

To prove Theorem \ref{Thm1} in an arbitrary degree $n$, we need to show that
the above is compatible with the shift functors $\Sigma_A$ and $\Sigma_B$
on the relatively $\CO$-stable categories $\modbar(A)$ and $\modbar(B)$.
For simplicity, we denote both shift functors by $\Sigma$. 
We have $\hatExt^n_A(U,V) = \Hombar_A(U,\Sigma^n(V))$, and
$\hatExt^{-n}_A(V,U)=\Hombar_A(V,\Sigma^{-n}(U))\cong \Hombar_A(\Sigma^n(V),U)$,
where the second isomorphism is obtained from applying the functor $\Sigma^n$.
The Tate duality 
$$\langle - , - \rangle_A :  \hatExt_A^{n}(U,V)\times \hatExt_A^{-n}(V,U)
\to K/\CO$$ 
is induced by the map sending $(\alpha, \gamma)\in$
$\Hom_A(U,\Sigma^n(V))\times \Hom_A(V, \Sigma^{-n}(U))$ to the trace on $KU$ of
the endomorphism 
$$z_A^{-1} \Sigma^n(\gamma) \circ \alpha;$$
in other words, this is induced by the degree zero duality applied to $U$ and $\Sigma^n(V)$,
combined with the shift functor $\Sigma^n$.
Applying the degree zero case to $U$ and $\Sigma^n(V)$  yields the equation
$$\langle \alpha, \tr_M(\Sigma^n(\beta))\rangle_A  = 
\langle \Id_{M^\vee}\ten\alpha, \Sigma^n(\beta) \rangle_B$$
in $K/\CO$.
It remains to show that the left side is equal to
$\langle \alpha, \Sigma^n(\tr_M(\beta)) \rangle_A$.
This expression depends only on the images of the morphisms $\alpha$, $\beta$
in their respective stable categories. By \ref{transferSigmacommute}, 
the images in the stable category
$\modbar(A)$ of $\Sigma^n(\tr_M(\beta)$ and $\tr_M(\Sigma^n(\beta))$ are
equal. Again  using \ref{Tate-symmetric}, the result follows.
\end{proof}

\section{Proof of Theorem \ref{Thm2}}

As in the previous section, let
$\CO$ be a complete discrete valuation ring with field of fractions $K$ of 
characteristic zero. Let $A$, $B$ be symmetric $\CO$-algebras such that
$KA=K\tenO A$ and $KB=K\tenO B$ are semisimple. Fix symmetrising forms $s$, $t$
of $A$, $B$, respectively. Let $M$ be an $A$-$B$-bimodule which is 
finitely generated projective as a left $A$-module and as a right $B$-module.
As before, we set $A^e=A\tenO A^\op$ and $B^e=B\tenO B^\op$. 
Let $n$ be an integer. 
The Tate duality
$$\langle - , - \rangle_{A^e} : \hatHH^n(A) \times \hatHH^{-n}(A) \to K/\CO$$
is induced by a map
$$\varphi_{A^e} : \Hom_{A^e}(A,\Sigma^n(A)) \times \Hom_{A^e}(A,\Sigma^{-n}(A))
\to K$$
 sending $(\alpha,\beta)\in$
$\Hom_{A^e}(A,\Sigma^n(A)) \times \Hom_{A^e}(A,\Sigma^{-n}(A))$
to the trace on $KA$ of the $A^e$-endomorphism 
$$z_A^{-2} \cdot \Sigma^n(\beta)\circ\alpha$$
where we use that the projective element $z_{A^e}=z_A\ten z_A$ acts as
multiplication by $z_A^2$ on the $A^e$-module $A$. Note that $\varphi_A$
depends on the choices of $\Sigma^n(A)$ and  $\Sigma^n(\beta)$, but the 
induced map to $K/\CO$ does not. 
We have the analogous
description for $B$ instead of $A$.  

As briefly described in Section \ref{prelim}, 
the functor $\Sigma^n$ on $\modbar(A^e)$ preserves the full subcategory
$\perfbar(A)$ of perfect  $A$-$A$-bimodules. Slightly more generally, the functor
$\Sigma^n$ on $\modbar(A\tenO B^\op)$ preserves $\perfbar(A,B)$.
Setting $X=$ $\Sigma^n(A)$, the
functor $\Sigma$ on $\modbar(A^e)$ restricted to $\perfbar(A)$ is canonically 
 isomorphic to the functor
induced by any of the two exact functors $X\tenA -$ and $-\tenA X$ on $\perf(A)$,
and the functor $\Sigma^n$ on $\modbar(A\tenO B^\op)$ is canonically isomorphic
to the functor induced by the exact functor $X\tenA-$ on $\perf(A,B)$. 
Similarly, setting $Y=\Sigma^n_{B^e}(B)$, the functor $\Sigma$ restricted to
$\perfbar(B)$ is canonically  isomorphic  to any of the two exact functors induced by
$Y\tenB-$ and $-\tenB Y$ on $\perf(B)$, and the functor $\Sigma^n$ on 
$\modbar(A\tenO B^\op)$ is canonically isomorphic to the functor induced
by the exact functor $-\tenB Y$ on $\perf(A,B)$. With this notation, we will
need the identification
\begin{Equation} \label{X1}
$$X\tenA M= \Sigma^n_{A\tenO B^\op}(M) = M\tenB Y, $$
\end{Equation}
in $\perfbar(A,B)$. 
Denote by 
\begin{Equation} \label{xi}
$$ \xi : X \tenA M \longrightarrow M \tenB Y$$
\end{Equation}
and $A$-$B$-bimodule homomorphism which induces the identification
in \ref{X1}. Since $\xi$ induces an isomorphism in $\perfbar(A,B)$, the kernel
and cokernel of $\xi$ are projective $A\tenO B^\op$-modules.

\begin{proof}[{Proof of Theorem \ref{Thm2}}]
Let $\zeta\in\hatHH^n(A)$, and $\tau\in \hatHH^{-n}(B)$.
Represent these classes by bimodule homomorphisms, abusively denoted
by the same letters,
$$\zeta : A \to \Sigma^n(A) = X,\ \ \ \ \ \tau : B \to \Sigma^{-n}(B).$$
Then $\Sigma^n(\tau)$ is represented by a morphism, again denoted by the
same letter,
$$\Sigma^n(\tau) : \Sigma^n(B) = Y \to B$$
where we have used the identification  $\Sigma^n(\Sigma^{-n}(B))=B$ in $\perfbar(B)$. 
We need to show that the trace on $KA$ of
$$z_A^{-2} \cdot (\Sigma^n(\tr_M(\tau))\circ \zeta)$$
is equal to the trace on $KB$ of
$$z_B^{-2} \cdot (\Sigma^n(\tau)\circ \tr_{M^\vee}(\zeta)).$$
For simplicity, all identity homomorphisms on any of the bimodules $M$, 
$M^\vee$, $X$, $Y$  are denoted $\Id$. 
The $A^e$-homomorphism $\tr_M(\tau)$ is equal to the composition
$$\xymatrix{A \ar[rr]^(.35){\epsilon_{M^\vee}} & &M\tenB B \tenB M^\vee
\ar[rr]^(.47){\Id \ten \tau \ten \Id}  & & M\tenB \Sigma^{-n}(B) \tenB M^\vee 
\ar[rr]^(.65){\Sigma^{-n}(\eta_M)}  &&  \Sigma^{-n}(A)}$$
where we have identified $M\tenB \Sigma^n(B) \tenB M^\vee$ and 
$\Sigma^n(M\tenB M^\vee)$ along a bimodule homomorphism inducing the
canonical isomorphism in the stable module category $\modbar(A^e)$. 
Thus $\Sigma^n(\tr_M(\tau))\circ \zeta$ is the composition
$$\xymatrix{A \ar[r]^(.45)\zeta& \Sigma^n(A) \ar[rr]^(.42){\Sigma^n(\epsilon_{M^\vee})} 
 &&\Sigma^n(M \tenB M^\vee)
\ar[rrr]^(.53){\Id \ten \Sigma^n(\tau) \ten \Id}  && & M\tenB  M^\vee 
\ar[rr]^(.58){\eta_M}  &  & A}$$
where in the third term we use the identification $\Sigma^n(M\tenB M^\vee)=$
$M\tenB \Sigma^n(B) \tenB M^\vee$ in $\modbar(A^e)$, and in the fourth term
we identify $M\tenB M^\vee=$ $M\tenB B\tenB M^\vee$. In terms of the bimodules
$X$ and $Y$, as well as replacing $\Sigma^n$ by $X\tenA-$ as appropriate,
this shows that  $\Sigma^n(\tr_M(\tau))\circ \zeta$ is represented by the composition
of morphisms in $\perf(A)$
\begin{Equation} \label{X1A}
$$\xymatrix{A \ar[r]^\zeta & X \ar[r]^(.30){\Id \ten\epsilon_{M^\vee}} &  
X\tenA M\tenB M^\vee \ar[r]^{\xi\ten\Id} & 
M\tenB Y\tenB M^\vee \ar[rr]^(.55){\Id\ten\Sigma^n(\tau)\ten\Id} & & 
M\tenB M^\vee \ar[r]^(.65){\eta_M}  & A}$$
\end{Equation}
Similarly,  the map 
$\Sigma^n(\tau)\circ \tr_{M^\vee}(\zeta)$ is represented by the composition
\begin{Equation} \label{Y1B}
$$\xymatrix{ B \ar[r]^(.40){\epsilon_M} &  M^\vee\tenA M \ar[rr]^(.45){\Id\ten\zeta\ten\Id}
& & M^\vee \tenA X \tenA M \ar[r]^{\Id\ten\xi} & M^\vee \tenA M\tenB Y  
\ar[r]^(.65){\eta_{M^\vee}}  & Y \ar[r]^{\Sigma^n(\tau)} & B}$$
\end{Equation}

Since we need to compare traces of endomorphisms of $KA$, $KB$, we may
extend coefficients to any field extension of $K$. Relative projective
elements are compatible with these coefficient extensions, and hence we may assume
that $\CO=K$ is a splitting field of $A$ and $B$. Thus we may assume  that $A$, $B$ are
finite direct products of matrix algebras over $K$. What we need to show is that
the traces of the two maps in \ref{X1A} and \ref{Y1B} multiplied by $z_A^{-1}$ and
$z_B^{-2}$, respectively, are equal.
Since traces are additive and  all maps above (in particular, the
adjunction maps, hence relative projective elements) are compatible with the 
block decompositions of $A$, $B$,
we may assume that $A=$ $\End_K(U)$ and $B=\End_K(V)$ for some
finite-dimensional $K$-vector spaces $U$, $V$. The adjunction  maps are
also  additive in $M$, so we may assume that $M=U\tenK V^\vee$ (this is,
up to isomorphism, the unique finite-dimensional  indecomposable 
$A$-$B$-bimodule).  The result follows from Proposition
\ref{Tate-matrix-2} with $\Sigma^n(\tau)$ instead of $\sigma$.
\end{proof}

\section{Remarks on transfer for Calabi-Yau triangulated categories} \label{CY-section}

There are two ways to associate transfer maps to a triangle functor
$\CF : \CC\to \CD$ of Calabi-Yau categories: either by making use of a
biadjoint functor $\CG$ (if there is such a functor, as described in 
many sources such as \cite{Chouinard}, \cite{Ligrblock}), or by making use of Serre duality.
Theorem \ref{Thm1} and \cite[Theorem 1.1]{LiHHtt} state that
for stable categories of symmetric algebras over fields or complete discrete
valuation rings  these two constructions  coincide.

\medskip
To be more precise, let
$\CC$, $\CD$ be triangulated categories over a field $k$. We use the same letter
$\Sigma$ for the shift functors in $\CC$ and $\CD$.  Suppose that 
homomorphism spaces between objects in either category  are finite-dimensional
and that $\CC$, $\CD$ admit Serre functors $\bS$, $\bT$, respectively.
Let $\CF : \CC\to \CD$ be a $k$-linear functor. Let $X$, $Y$ be objects in $\CC$.
Dualising the map
$$\Hom_\CC(Y,X) \to \Hom_\CD(\CF(Y), \CF(X))$$
induced by $\CF$, and making use of the defining property of a Serre functor,
yields a map
$$\tr_{\CC, \CD} : \Hom_\CD(\CF(X), \bT(\CF(Y))) \to \Hom_\CC(X, \bS(Y)) .$$
which makes the following diagram commutative.
\begin{Equation}\label{Serre-diagram}
$$\xymatrix{\Hom_\CD(\CF(X), \bT(\CF(Y))) \ar[rr]^{\tr_{\CC,\CD}} \ar[d]_{\cong}  & 
& \Hom_\CC(X,\bS(Y)) \ar[d]^{\cong} \\
\Hom_\CD(\CF(Y), \CF(X)))^\vee \ar[rr]_{\CF^\vee} & 
& \Hom_\CC(Y,  X)^\vee} $$
\end{Equation}
where $\CF^\vee$ is the dual of the map induces by $\CF$ on morphisms, and
where the vertical isomorphism are Serre duality isomorphisms.

\medskip
Assume now that  $\CC$ and $\CD$ are $d$-Calabi-Yau triangulated categories
 for some integer $d$. That is, the Serre functors  $\bS$, $\bT$ are isomorphic to 
 $\Sigma^d$ on $\CC$, $\CD$, respectively
(see Kontsevich  \cite{Konts98} or also Keller \cite{KellerCY} for background 
material and a long list  of references on Calabi-Yau triangulated categories - what 
we call Calabi-Yau would be called weakly Calabi-Yau in many sources).
Then the previous map $\tr_{\CC,\CD}$  takes the form
\begin{Equation} \label{trFvee-map}
$$\tr_{\CF^\vee} : \Ext^d_\CD(\CF(X), \CF(Y))=\Hom_\CD(\CF(X), \Sigma^d(\CF(Y))) \to
 \Hom_\CC(X, \Sigma^d(Y)) = \Ext^d_\CC(X,Y).$$
 \end{Equation}
If $\CF$ is a functor of triangulated categories, then the functors
$\Sigma\circ\CF$ and $\CF\circ\Sigma$ are isomorphic, so upon replacing $Y$ by
$\Sigma^{n-d}(Y)$, where $n$ is an integer, we get in particular  a map
\begin{Equation} \label{trFvee-map-n-d}
$$\tr_{\CF^\vee} : \Ext^n_\CD(\CF(X), \CF(Y))=\Hom_\CD(\CF(X), \Sigma^n(\CF(Y))) \to
 \Hom_\CC(X, \Sigma^n(Y)) = \Ext^n_\CC(X,Y).$$
 \end{Equation}
For $n=0$ this yields a map
\begin{Equation} \label{trFvee-0}
$$\tr_{\CF^\vee} : \Hom_\CD(\CF(X), \CF(Y)) \to \Hom_\CC(X,Y).$$
\end{Equation}
Combining the above, the map
$\tr_{\CF^\vee}$ makes the following diagram commutative:
\begin{Equation}\label{CY-diagram}
$$\xymatrix{\Hom_\CD(\CF(X), \CF(Y)) \ar[rr]^{\tr_{\CF^\vee}} \ar[d]_{\cong}  & 
& \Hom_\CC(X,Y) \ar[dd]^{\cong} \\
\Hom_\CD(\CF(Y),\Sigma^d(\CF(X)))^\vee \ar[d]_{\cong} & & \\
\Hom_\CD(\CF(Y), \CF(\Sigma^d(X)))^\vee \ar[rr]_{\CF^\vee} & 
& \Hom_\CC(Y, \Sigma^d(X))^\vee} $$
\end{Equation}
Here the vertical isomorphisms are given by a choice of Calabi-Yau duality 
together with a choice of an isomorphism $\Sigma\circ\CF\cong$ 
$\CF\circ\Sigma$,  and the bottom horizontal is the dual of the map 
induced by $\CF$. (Note the dependence on choices.)

If $\CF$ has a biadjoint functor of triangulated categories $\CG  : \CD\to \CC$,
then (following e. g. \cite{Chouinard} or \cite[\S 4]{Ligrblock}) we also have a transfer map 
\begin{Equation}\label{transfer-map}
$$\tr_\CG : \Hom_\CD(\CF(X),\CF(Y)) \longrightarrow \Hom_\CC(X,Y)$$
\end{Equation}
sending a morphism $\psi : \CF(X)\to \CF(Y)$ to the composition of
morphisms 
$$\xymatrix{X \ar[r] & \CG(\CF(X)) \ar[r]^{\CG(\psi)} & \CG(\CF(Y)) \ar[r] & Y},$$
where the first map is the adjunction unit of $\CF$ being left adjoint to $\CG$,
and the last map is the adjunction counit of $\CF$ being right adjoint to $\CG$.
The map $\tr_\CG$ depends on the choice of adjunction isomorphisms. 

\begin{Remark}
It would be desirable to spell out the exact compatibility conditions  for adjunction
isomorphisms and Calabi-Yau duality that would lead to an equality of the
maps $\tr_{\CF^\vee}=\tr_\CG$ in the equations \ref{trFvee-0} and \ref{transfer-map}.
\end{Remark}

One can rephrase  \cite[Theorem 1.2]{LiHHtt} as stating
that $\tr_{\CF^\vee}=\tr_\CG$. The proof amounts to showing that the choices 
made for adjunction isomorphisms
and Tate duality determined by the choices of symmetrising forms are compatible.
Indeed,  if $A$ is a  symmetric
$k$-algebra, then  Tate duality turns $\modbar(A)$ into a $(-1)$-Calabi-Yau 
triangulated category.
Given two symmetric $k$-algebras $A$, $B$ and an $A$-$B$-bimodule $M$ in
$\perf(A,B)$,  
 Theorem \cite[Theorem 1.2]{LiHHtt}  can be rephrased as stating
 that the transfer maps $\tr_M(U,V)$ are special cases of the construction 
of $\tr_{\CF^\vee}$ in the above diagram \ref{CY-diagram}; in other words, the
maps in \ref{trFvee-0} and \ref{transfer-map} coincide. 

\medskip
Let now $A$ be a symmetric algebra over a complete discrete valuation ring $\CO$ with
a field of fractions $K$ of characteristic $0$ such that $K\tenO A$ is semisimple.
Extend the notion of Calabi-Yau triangulated
categories to $\CO$-linear triangulated categories by using duality with respect to 
the Matlis module $K/\CO$.  Note that $K/\CO$ is the degree $1$-term of the
injective resolution $K\to K/\CO$ of $\CO$. Note further that there are no 
nonzero $\CO$-linear maps from the torsion $\CO$-modules $\Hombar_A(U,V)$ to 
the torsion free $\CO$-module $K$ (where $U$, $V$ are $A$-lattices),
and hence Matlis duality coincides with $\RHom(-, \CO)$ on the morphism spaces
in the stable module category, except  for a degree shift. 

\begin{Remark}
Tate duality on the stable module category 
$\underline{\mathrm{latt}}(A)$ of
finitely generated $\CO$-free $A$-modules as described in \cite[Theorem 1.3]{EGKL}
seems to suggest that $\underline{\mathrm{latt}}(A)$ should be called $0$-Calabi-Yau. 
If, however, one were to take  into account that $K/\CO$ is in degree $1$, then the total 
degree of Tate duality $\Hombar_A(U,V)\times \Hombar_A(V,U)\to K/\CO$ would again 
be $-1$.
\end{Remark}

Regardless of dimension considerations, Theorem \ref{Thm1} shows that the
transfer maps $\tr_M(U,V)$ in Theorem \ref{Thm1} 
 are special cases of the construction given by the diagram \ref{CY-diagram} with
 $k$-duality replaced by $K/\CO$-duality.  As pointed out earlier, there are choices
to be made: Tate duality depends on the choices of symmetrising forms, and showing
that the maps from \ref{trFvee-0} and \ref{transfer-map} are equal in Theorem
\ref{Thm1} boils down to being able to make compatible choices.

\begin{Remark}
In the context of stable module categories of symmetric algebras 
there is always a canonical choice for the commutation with shift functors.
More precisely, given two symmetric $\CO$-algebras $A$, $B$  
and an $A$-$B$-bimodule $M$ in $\perf(A,B)$, then
the  functor $M\tenB-$ from $\mod(B)$ to $\mod(A)$ is exact, preserves projectives,
hence preserves projective resolutions, and therefore induces a canonical isomorphism
$\CF\circ\Sigma\cong\Sigma\circ\CF$, where $\CF : \underline{\mathrm{latt}}(B)\to$
$\underline{\mathrm{latt}}(A)$ is the functor induced by $M\tenB-$.
\end{Remark}

\section{On products in negative degrees of Tate cohomology}

Nonzero products in negative degree in Tate cohomology have implications for the
depth of the nonnegative part. This phenomenon was  first observed in \cite{BeCaTate}
in finite group cohomology, and then generalised in \cite{BJO}, \cite[\S 8]{LiHHtt}.
This arises  over complete discrete valuation rings as well, with essentially the
same arguments used in \cite{BeCaTate}.

Let $A$ be a symmetric algebra over a complete discrete valuation ring $\CO$ with
a field of fractions $K$ of characteristic zero. Let $U$, $V$, $W$ be finitely
generated $A$-modules. Let $m$, $n$ be integers.
 Let $\alpha\in\hatExt^m_A(U,V)$, $\beta\in\hatExt^n_A(V,W)$,
and $\gamma\in$ $\hatExt^{-m-n}_A(W,U)$. We denote by $\beta\alpha\in$
$\hatExt^{m+n}_A(U,W)$ the Yoneda product; that is, $\beta\alpha$ is represented
by $\Sigma^m(\beta)\circ\alpha$, where we use the same letters $\alpha$, $\beta$
for representatives of their classes in $\Hom_A(U,\Sigma^m(V))$, $\Hom_A(V,\Sigma^n(W))$. 
We have
\begin{Equation} \label{Tateassociative}
$$\langle \beta\alpha, \gamma\rangle_A = \langle \alpha, \gamma\beta\rangle_A$$
\end{Equation}
because both sides are equal to the image in $K/\CO$ of the  trace on $KU$ of the
endomorphism
$z_A^{-1}\cdot(\Sigma^{m+n}(\gamma)\circ \Sigma^m(\beta)\circ\alpha)$ of $KU$, 
where as before
we use abusively the same letters $\alpha$, $\beta$, $\gamma$ for representatives
in $\Hom_A(U,\Sigma^m(V))$, $\Hom_A(V,\Sigma^n(W))$, $\Hom_A(W, \Sigma^{-m-n}(U))$
of their classes.
Applied with $U=W$ and $\gamma=\Id_U$  this yields
\begin{Equation} \label{Tateassociative-1}
$$\langle \beta\alpha, \Id_U\rangle_A = \langle \alpha, \beta\rangle_A.$$
\end{Equation}

\begin{Lemma} \label{negTate1}
Let $\zeta$ be nonzero element in $\hatExt^n_A(U,V)$. Then there is a nonzero
element $\eta$ in $\hatExt^{-n}_A(V,U)$ such that the Yoneda product $\zeta\eta$
is nonzero in $\hatExt^0_A(U,U)=\Endbar_A(U)$.
\end{Lemma}

\begin{proof}
By Tate duality, there is $\eta\in$ $\hatExt^{-n}_A(V,U)$ such that 
$\langle \zeta, \eta\rangle_A\neq$ $0$. By \ref{Tateassociative-1} we have
$\langle \eta\zeta, \Id_U\rangle\neq0$, and hence $\eta\zeta\neq 0$ and
$\eta\neq 0$.
\end{proof}

\begin{Remark} \label{depth1Remark}
As in \cite[\S 8]{LiHHtt}, we denote by $\barExt^*(U,U)$ the nonnegative part of 
$\hatExt^*_A(U,U)$. Adapting the arguments from \cite{BeCaTate}, as reproduced in
the proof of \cite[Proposition 8.3]{LiHHtt}, shows that if $\hatExt^*_A(U,U)$ has a
nonzero product of two homogeneous elements in negative degrees, and if
$\barExt^*_A(U,U)$ is graded-commutative, then $\barExt^*_A(U,U)$ does not
have a regular sequence of length $2$. In particular, if $\hatHH^*(A)$ has a
nonzero product of two homogeneous elements in negative degrees, then its
nonnegative part $\barHH^*(A)$ does not have a regular sequence of length $2$.
\end{Remark}



\begin{thebibliography}{WWW}

\bibitem{BenNY} D. Benson, {\em Modules with injective cohomology, and 
local duality for a finite group}, New York J. Math. {\bf 7} (2001), 
201--215.

\bibitem{BeCaTate} D. J. Benson and Jon F. Carlson, {\it Products in
negative cohomology}. J. Pure Appl. Alg. {\bf 82} (1992), 107--129.


\bibitem{BJO} P.A. Bergh, D. A. Jorgensen, and S. Oppermann, 
{\it The negative side of cohomology for Calabi-Yau
categories}. Bull. London Math. Soc.{\bf 46} (2014), 291--304.

\bibitem{Bocklandt} R. Bocklandt, {\em Graded Calabi-Yau algebras of 
dimension $3$.} (Appendix by M. Van den Bergh.) J. Pure Appl. Algebra
{\bf 212} (2008), 14-32.

\bibitem{Broue1} M. Brou\'e, {\em On representations of  %
symmetric algebras: an introduction}, Notes by M. Stricker, 
Mathematik Department ETH Z\"urich (1991).

\bibitem{Broue2} M. Brou\'e, {\em Higman's criterion revisited}, %
Michigan Math. J. {\bf 58} (2009), 125--179.

\bibitem{BuchHan} R.-O. Buchweitz, {\em Maximal Cohen-Macaulay modules
and Tate-cohomology over Gorenstein rings}, manuscript, Hannover (1986).

\bibitem{Chouinard} L. G. Chouinard, {\em Transfer maps},
Comm. Alg. {\bf 8} (1980), 1519--1537.

\bibitem{Dugas12} A. Dugas, {\em Resolutions of mesh algebras: periodicity
and Calabi-Yau dimensions}. Math. Z. {\bf 271} (2021), 1151--1184.

\bibitem{EGKL} F. Eisele, M. Geline, R. Kessar, and M. Linckelmann,
{\em On Tate duality and a projective scalar property for symmetric
algebras}.   Pacific J. Math.
{\bf 293} (2018), 277--300.



\bibitem{ErdSkow} K. Erdmann and A. Skowro\'nski, {\em The stable
Calabi-Yau dimension of tame symmetric algebras.} J. Math. Soc. Japan {\bf 58}
(2006), 97--123.

\bibitem{Farrell} F. T. Farrell, {\em An extension of Tate cohomology to a
class of infinite groups}, J. Pure Appl. Algebra {\bf 10} (1977), 153--161.

\bibitem{Goichot} F. Goichot, {\em Homologie de Tate-Vogel \'equivariante},
J. Pure Applied Algebra {\bf 82} (1992), 39--64.

\bibitem{Greenlees} J. P. C. Greenlees, {\em Commutative algebra in group 
cohomology}, J. Pure Appl. Algebra {\bf 98} (1995), 151--162.

\bibitem{Ivanov13} S. O. Ivanov, {\em Self-injective algebras of stable
Calabi-Yau dimension three}. J. Math. Sci. (N.Y.) {\bf 188} (2013), 601--620.

\bibitem{IvaVol14} S. O. Ivanov and Y. V. Volkov, {\em Stable Calabi-Yau dimension
of self-injective algebras of finite type.} J. Algebra {\bf 413} (2014), 72--99.

\bibitem{KellerCY} B. Keller, {\em Calabi-Yau triangulated categories.} In: Trends in
Representation Theory of Algebras and Related Topics, EMS Ser. Congr. Rep, Eur. Math.
Soc. Z\"urich (2008), 467--489.

\bibitem{Konts98} M. Kontsevich, {\em Triangulated categories and geometry.}
Course at the \'Ecole Normale Sup\'erieure, Paris. Notes taken by J. Bellaiche, J-F. Dat,
I. Marin, G. Racinet, and H. Randriambololona (1998).

\bibitem{KLcm} R. Kessar and M. Linckelmann, {\em The Castelnuovo-Mumford
regularity of the cohomology of fusion systems and of the Hochschild
cohomology of block algebras is zero}, preprint (2013).

\bibitem{Lintransfer} M. Linckelmann,
{\em Transfer in Hochschild cohomology of blocks of finite groups},
Algebras Representation Theory \textbf{ 2} (1999), 107--135.

\bibitem{Ligrblock} M. Linckelmann, {\em On graded centers and block   %
cohomology}, Proc. Edinb. Math. Soc. (2) {\bf 52} (2009), 489--514.



\bibitem{LiHHtt} M. Linckelmann,                                                                    %
{\em Tate duality and transfer in Hochschild cohomology}, J. Pure
Appl. Algebra {\bf 217} (2013), 2387--2399.



\bibitem{LiBookI} M. Linckelmann, {\em The block theory of finite group %
algebras I}, Cambridge University Press, London Math. Soc. Student
Texts {\bf 91} (2018). 

\bibitem{LiBookII} M. Linckelmann, {\em The block theory of finite group
algebras II}, Cambridge University Press, London Math. Soc. Student
Texts {\bf 92} (2018). 



\bibitem{Morita65} K. Morita, {\em Adjoint pairs of functors and Frobenius 
extensions}. Sci. Rep. Tokyo Kyoiku Daigaku Sect. A (1965), 40-71.


\bibitem{Thev} J. Th\'evenaz, {\em $G$-Algebras and Modular
Representation Theory}, Oxford Science Publications, Clarendon,
Oxford (1995).

\end{thebibliography}
\end{document}